\documentclass{amsart}

\usepackage{amscd}
\usepackage{amssymb, amsfonts}
\usepackage{amsrefs}
\usepackage[T1]{fontenc}
\usepackage{tikz}
\usepackage{tikz-cd}
\usepackage{hyperref}

\newcommand{\Q}{{\mathbb Q}}
\newcommand{\F}{{\mathbb F}}
\newcommand{\Z}{{\mathbb Z}}

\newcommand{\cF}{\mathcal{F}}
\newcommand{\cG}{\mathcal{G}}
\newcommand{\ri}{\mathcal{O}}

\DeclareMathOperator{\gal}{Gal}

\DeclareMathOperator{\vol}{vol}
\DeclareMathOperator{\loc}{Loc}
\DeclareMathOperator{\ind}{c-Ind}
\DeclareMathOperator{\ord}{ord}
\DeclareMathOperator{\Hom}{Hom}

\newcommand{\GL}{\mathbf {GL}}
\newcommand{\bG}{\mathbf{G}}
\newcommand{\bT}{\mathbf {T}}
\newcommand{\bB}{\mathbf {B}}

\newcommand{\bR}{\mathbf{R}}

\newcommand{\can}{\mathrm{can}}
\newcommand{\ff}{{\mathfrak f}}

\newcommand\lef{\mathbb L}
\newcommand\cE{{\mathcal E}}
\newcommand\mot{\mathrm{mot}}
\newcommand\ad{\mathrm{ad}}
\newcommand{\Ner}[1]{\mathcal{#1}}
\newcommand{\NerC}[1]{\mathcal{#1}^\circ}

\def\llp{\mathopen{(\!(}}
\def\rrp{\mathopen{)\!)}}

\theoremstyle{plain}
\newtheorem{thm}{Theorem}
\newtheorem{theorem}[thm]{Theorem}
\newtheorem{lem}[thm]{Lemma}
\newtheorem{cor}[thm]{Corollary}
\newtheorem{prop}[thm]{Proposition}

\theoremstyle{definition}
\newtheorem{rem}[thm]{Remark}

\newtheorem{example}[thm]{Example}

\title[]{The canonical measure on a reductive $p$-adic group is motivic}
\author{Julia Gordon and David Roe}
\subjclass[2010]{22E50 (primary); 20G25, 14E18, 03C60 (secondary)}
\keywords{N\'eron models, Denef-Pas language, algebraic tori, canonical measure}

\begin{document}

\begin{abstract}  Let $G$ be a connected reductive group over a non-Archimedean local field. 
We prove that its parahoric subgroups are definable in the Denef-Pas language, which is a
first-order language of logic used in the theory of motivic integration developed by Cluckers and Loeser.
The main technical result is the definability of the connected component of the N\'eron model
of a tamely ramified algebraic torus.  As a corollary, we prove that the canonical Haar measure on $G$,
which assigns volume $1$ to the particular \emph{canonical} maximal parahoric defined by
Gross in \cite{gross:97a}, is motivic. This result resolves a technical difficulty that arose in
\cite{cluckers-gordon-halupczok:14b} and \cite{shin-templier:15a}*{Appendix B} and permits a simplification of some of the proofs in those articles. 
It also allows us to show that formal degree of a compactly induced representation is a
motivic function of the parameters defining the representation. 
\end{abstract}
\maketitle

\section{Introduction}
The goal of this paper is to complete a technical step in the 
definable formulation of 
the representation theory of $p$-adic groups, a project started by T.C. Hales in 1999. 
Here, the word ``definable'' is as in the theory of motivic integration developed by R. Cluckers and F. Loeser \cite{cluckers-loeser:08a}.

Specifically, we will prove that the parahoric subgroups of a connected reductive $p$-adic group
are definable using the Denef-Pas language, which is the language used in the Cluckers-Loeser
theory of motivic integration and its applications to representation theory of $p$-adic groups.
As a consequence, we prove that the canonical Haar measure on a connected reductive group 
(which assigns the volume $1$ to the canonical parahoric subgroup constructed by B. Gross \cite{gross:97a}) is motivic.

For unramified groups, this statement has been known for a while \cite{cluckers-hales-loeser}.
For ramified groups, the definition of the canonical smooth model of $\bG$ relies on the
N\'eron model of a maximal torus in $\bG$, which does not behave well with respect to Galois descent.
The main technical result of this paper is that the connected component of the N\'eron model
of a tamely ramified torus is definable in the language of Denef-Pas. The difficulty in proving this result
is caused by the fact that ``taking the connected component'' is not an operation that can be easily
described by first-order logic. 

This paper is split into two sections, the first leading up to Proposition \ref{prop:NerCdefinable},
which shows that the connected component of the N\'eron model of a torus is definable,
and the second giving applications to canonical measures and formal degrees.

We begin Section \ref{sec:tori} by setting up notation and briefly reviewing the Denef-Pas language.  In order
to give formulas defining $\bT(F)$ and $\NerC{T}(\ri_F)$  (where $\NerC{T}$ denotes the connected
component of the N\'eron model of the torus $\bT$) in this language, we need to parameterize the possible tori $\bT$.
In Section \ref{sub:fixedchoices}, we describe the choices that can be made without reference to variables in $F$,
such as fixing an abstract Galois group $\Gamma$ and a lattice with action of $\Gamma$ which will play the role of the cocharacter lattice of $\bT$.
Section \ref{sub:def_tori} then parameterizes field extensions with Galois group $\Gamma$, resulting in a parameterization
of tori over $F$.  Finally, in Section \ref{sub:NerCdefinable} we show that $\NerC{T}(\ri_F)$ is a definable subgroup of $\bT(F)$.
In Section 3 we prove two easy corollaries mentioned above, namely, that the canonical measure is motivic,
and in a definable family of compactly-induced irreducible representations, formal degree is motivic. 

{\bf Acknowledgment.} We thank Loren Spice for the discussion of formal degree, and the referee for multiple helpful suggestions and corrections.

\section{Tori} \label{sec:tori}
We will use the notions of definable sets and definable functions, which will always refer to the Denef-Pas language. 
Formulas in the Denef-Pas language can have variables of three \emph{sorts}: valued field (which will be denoted by \emph{VF}),
residue field (denoted by \emph{RF}) and the value group. Even though we will often be working with ramified extensions,
we always start with a local field $F$ with normalized valuation, so the value group is $\Z$
(the \emph{VF}-variables will range over $F$, and so their valuations will be in $\Z$).
Formulas in the Denef-Pas language can be interpreted given a choice of a valued field \emph{together with a uniformizer}. 
We refer the reader to \cite{cluckers-gordon-halupczok:14c} and references therein for the definitions of the Denef-Pas language, definable sets, etc. 

For us, $F$ will always be a non-Archimedean local field: either $\F_q\llp t\rrp$ or a finite extension of $\Q_p$.
As a consequence of the definition of a definable set, all statements in this paper will hold for any $F$ of sufficiently large residue characteristic $p$, 
though we will give no effective bound on $p$. 
Given an integer $M>0$, we will denote by $\loc_M$ the collection of non-Archimedean local fields
with residue characteristic greater than $M$. 

For a local field $F$, we will denote its ring of integers by $\ri_F$, its residue field by $k_F$,
and let $q_F=\# k_F$. The symbol $\varpi$ or $\varpi_F$ will stand for the uniformizer of the valuation on $F$. 
A formula in the Denef-Pas language  with $n$ free \emph{VF}-variables, $m$ free \emph{RF}-variables, and $r$ free 
$\Z$-variables 
defines a subset of $F^n\times k_F^m \times \Z^r$. 
We will denote the definable set $F^n\times k_F^m \times \Z^r$ itself by
$\mbox{\emph{VF}}^n\times \mbox{\emph{RF}}^m\times \Z^r$.
In earlier works on motivic integration this set was denoted by $h[n,m,r]$. 
We will talk about definable subsets of $\mbox{\emph{VF}}^n\times \mbox{\emph{RF}}^m\times \Z^r$,
meaning the subsets defined by Denef-Pas formulas with the right number of free variables of each sort, as above. 
For a definable subset $X\subset \mbox{\emph{VF}}^n\times \mbox{\emph{RF}}^m\times \Z^r$, and given a local field $F$,  
we will denote by $X(F)$ the \emph{specialization} of $X$ in $F$, i.e., the subset of
$F^n\times k_F^m\times \Z^r$ obtained by interpreting in $F$ all the formulas defining the set $X$.\footnote{This
is the notation used in \cite{gordon-hales:15a}; note that traditionally in the motivic integration literature,
the {specialization} of a definable set $X$ was denoted by $X_F$, but this notation generates too many subscripts for us.} 

We start by setting up the framework for working with tori in the Denef-Pas language,
following \cite{cluckers-hales-loeser}, \cite{cluckers-gordon-halupczok:14b} and \cite{gordon-hales:15a}.

\subsection{Fixed choices}\label{sub:fixedchoices}

As in \cite{gordon-hales:15a}*{\S 2.1}, we begin by outlining our \emph{fixed choices},
which are made before writing any formulas in the Denef-Pas language.  
For each \emph{fixed choice} (which will be completely field-independent), we will further describe
a definable set of parameters (which will then be allowed to range over a valued field $F$,
its residue field $k_F$ or $\Z$), in such a way that each tuple of parameters gives rise to an algebraic torus
defined over $F$, and all isomorphism classes of algebraic $F$-tori arise via this construction.

We fix a finite group $\Gamma$ and a normal subgroup $I \unlhd \Gamma$,
as well as enumerations of their elements $\Gamma = \{\sigma_1, \dots, \sigma_m\}$ and
$I = \{\sigma_1, \dots, \sigma_e\}$.  We make the convention that $\sigma_1 = 1$ and $\sigma_m$
generates $\Gamma / I$.\footnote{Note that we do not assume that $\sigma_m$ is the Frobenius element, since $p$ is not fixed.}
When we eventually construct a torus $\bT$ from the fixed choices
and corresponding parameters, these groups will play the roles of $\gal(E/F)$
and its inertia subgroup, where $E$ is the splitting field of $\bT$.

In order to define a torus $\bT$, we will use the equivalence of categories between $F$-tori and
free $\Z$-modules with a Galois action.  To this end, we fix a positive integer $n$ and an injective homomorphism
\begin{equation} \label{eq:theta}
\theta : \Gamma \hookrightarrow \GL_n(\Z),
\end{equation}
which gives $\Z^n$ an action of $\Gamma$.  The $\Gamma$-module $X$ defined by $\theta$ will play the role of $X_\ast(\bT)$.

Finally, we fix a resolution of $X$ by an induced $\Gamma$-module $Y$, i.e. a surjective map
$Y \to X$ where $Y$ has a basis permuted by $\Gamma$ (cf. \cite{brahm:thesis}*{Satz 0.4.4}).
To specify $Y$, we just give the matrix for the map $Y \to X$ of free abelian groups,
together with the matrices giving $\gamma : Y \to Y$ for $\gamma \in \Gamma$.
This resolution will allow us to definably cut out the connected component of the N\'eron model inside $\bT(F)$.

\subsection{Parameterizing field extensions and tori}\label{sub:def_tori}
 
We encode field extensions in the same way as in \cite{cluckers-gordon-halupczok:14b}.
Namely, we parameterize Galois extensions $E/F$ with $\gal(E/F) \cong \Gamma$
and realize all tori over $F$ that split over $E$ with cocharacter lattice $X$.
This parameterizes such tori as members of a family of definable sets,
for all $F$ of sufficiently large residue characteristic. 

We will write $L$ for the maximal unramified subextension of $E/F$.  In order to encode the data of the extension tower $E/L/F$,
we let $f=m/e$ and introduce parameters $b_0,\dots, b_{f-1}$, ranging over $\ri_F$.
We set $b(x)=x^f+b_{f-1}x^{f-1}+ \dots + b_0$. 
Similarly, we introduce parameters $c_0, \dots, c_{e-1}$, ranging over $L$
(i.e., each is given by an $f$-tuple of elements of $F$) and set $c(y) = y^e + c_{e-1}y^{e-1} + \dots + c_0$.

We impose the following conditions on these parameters, all of which are definable by formulas in the Denef-Pas language. 
\begin{enumerate}
\item The reduction of $b(x)$ modulo $\varpi_F$ is irreducible over $k_F$. 
This ensures that $F[x]/(b(x))$ is a degree $f$ unramified field extension of $F$. 
We denote this extension by $L$, and once and for all fix an identification with $F^f$ as 
an $F$-vector space. 
\item The polynomial $c(x)$ is Eisenstein: $\ord_L(c_0) = 1$ and $\ord_L(c_i) \ge 1$ for all $i$.
We further assume that the resulting extension $E = L[x]/(c(x))$ is Galois over $F$.
We fix an identification of $E$ with $L^e$ as $L$-vector spaces, and thus with $F^m$ as $F$-vector spaces. 
\item The field automorphisms of $E$ over $F$, as specified by $m \times m$ matrices over $F$, form a group isomorphic to $\Gamma$.
We will write $\sigma_i$ for the matrix as well as the corresponding element of $\Gamma$.
\item The automorphisms $\sigma_1, \dots, \sigma_e$ fix $L$, and the restriction of $\sigma_m$ to $L$ has order $f$.
\end{enumerate}
We denote by $\cE_\Gamma$ the space of parameters $(b_0, \dots, b_{f-1}, c_0, \dots, c_{e-1}, \sigma_1, \dots, \sigma_m)$ with these properties,
thought of as a definable subset of some large affine space over $F$.
For each local field $F$, every element of $\cE_\Gamma$ gives rise to a tower of field extensions $E/L/F$
with $\gal(E/F)$ isomorphic to $\Gamma$ and satisfying all the above conditions.
The homomorphism $\theta$ of \eqref{eq:theta} then defines a torus $\bT$ over $F$ with cocharacter lattice $X$ that splits over $E$.
More precisely, the set $E^\times \otimes X$ can be encoded as an open and definable subset
of an affine space over $F$ depending only on the fixed choices $X$ and $m$.
The group $\Gamma$ acts on $E^\times$ by means of the matrices $\sigma_i$ and on $X$
via the fixed choice $\theta$, and thus it acts on $E^\times \otimes X$ as well. 
This action is definable, in the sense that every element of $\Gamma$ acts by a definable map, and therefore the set
$\bT(F) = \bT(E)^\Gamma$ is definable as well.

Note that different parameters in $\cE_\Gamma$ may yield isomorphic extensions,
but that every isomorphism class of $E/F$ with Galois group $\Gamma$
arises from some element of $\cE_\Gamma$.  Moreover, as $\theta$ ranges over all homomorphisms
$\Gamma \hookrightarrow \GL_n(\Z)$,\footnote{There are infinitely many choices of $\theta$
but we never quantify over them; instead, we work with each such fixed choice separately.}
all possible cocharacter lattices of tori of dimension $n$ appear. Therefore every $F$-torus
arises via this construction, since it is determined by its splitting field and cocharacter lattice viewed as a $\Gamma$-module. 
\begin{example}
Suppose $\Gamma = I = \Z / 2\Z$ and $n=1$; note that $\theta$ is uniquely determined in this case.
For $p > 2$, the two ramified quadratic extensions of $F$ appear as members of the same family,
one for the polynomial $c(x) = x^2-\varpi$, another for the polynomial $c(x) = x^2-\epsilon\varpi$,
where $\epsilon \in \ri_F^\times$ is a non-square.
Recall that the interpretation of formulas in the Denef-Pas language depends not just on the field,
but also on the choice of uniformizer. In this case, a different choice of the uniformizer
would swap these two extensions, but both would still appear.
The torus $\bT$ is the the one-dimensional unitary group that splits over $E$.
\end{example}
\begin{rem}
While we have not constrained $\Gamma$ in such a way that $E/F$ is automatically tame,
if $\Gamma$ is not the semidirect product of two cyclic groups then
$\cE_\Gamma$ will be empty for large enough residue characteristic.  In this case, from the point of view
of motivic integration, $\cE_\Gamma$ is indistinguishable from the empty set.
\end{rem}

\subsection{The identity component of the N\'eron model} \label{sub:NerCdefinable}

Now that we have parameterized $E/L/F$ and $\bT$ and shown that $\bT(F)$ definable, we may prove the main technical result of the paper.
Write $\Ner{T}$ for the N\'eron model of $\bT$ (cf. \cite{bosch-lutkebohmert-reynaud:NeronModels}*{Ch. 10});
this is a model for $\bT$ over $\ri_F$ with the property that $\Ner{T}(\ri_F) = \bT(F)$.  Let $\NerC{T}$ be its identity component.

\begin{prop} \label{prop:NerCdefinable}
The subset $\NerC{T}(\ri_F) \subseteq \bT(F)$ is definable.
\end{prop}
\begin{proof}
We first reduce to the case that $L = F$: if $\NerC{T}(\ri_L)$ is definable then so is
$\NerC{T}(\ri_F) = \NerC{T}(\ri_L)^\Gamma$, where the equality holds since N\'eron models commute with unramified base change.
So for the remainder of the proof we will assume that $E/F$ is totally ramified.

Now, the identity component $\NerC{T}(\ri_F)$ is the kernel of the map $w_\bT : \bT(F) \to X_I$ from
$\Ner{T}(\ri_F)$ to its component group  defined in \cite{kottwitz:isocrystals-2}*{\S 7} (see also \cite{bitan:11a}*{3.1}).

Our fixed choice of resolution $Y \to X$ yields an induced torus $\bR$ over $F$ with cocharacter lattice $Y$, together with a diagram
\[
\begin{tikzcd}
\bR(F) \rar{\alpha} \dar{w_{\bR}} \drar{\beta} & \bT(F) \rar \dar{w_{\bT}} & 1 \\
Y_I \rar & X_I \rar & 0
\end{tikzcd}
\]
as in \cite{kottwitz:isocrystals-2}*{(7.2.6)}.  The map $\bR(F) = (E^\times \otimes Y)^I \to (E^\times \otimes X)^I = \bT(F)$
is definable since it is induced by the fixed map $Y \to X$.  Since $\bR$ is induced,
$Y_I$ is torsion free and $w_\bR : \bR(F) \to Y_I = \Hom(X^\ast(\bR), \Z)$ is given by
$r \mapsto \left(\lambda \mapsto \ord_{F}(\lambda(r))\right)$ \cite{kottwitz:isocrystals-2}*{(7.2.3)}.
Therefore $w_\bR$ is definable, and so is the composition $\beta : \bR(F) \to X_I$ with the fixed map $Y_I \to X_I$.

We can now show that $\NerC{T}(\ri_F)$ is a definable subset of $\bT(F)$: we have $t \in \NerC{T}(\ri_F)$
if and only if $\exists r \in \bR(F)$ such that $\alpha(r) = t$ and $\beta(r) = 0$.
\end{proof}

\section{General reductive groups}\label{sec:groups}
Let $\bG$ be a connected reductive algebraic group defined over a local field $F$,
let $\ff$ be a facet in the building of $\bG$ over $F$, and let $x$ be in the interior of $\ff$. 
We denote by $G_\ff$ the maximal parahoric subgroup $\bG(F)_{x, 0}$ associated with this data
by Moy and Prasad.
In this section we construct a family of definable sets that specialize to the parahoric subgroup ${G}_{\ff}$
for all fields $F$ of sufficiently large residue characteristic. By this we mean that first one constructs a family of definable sets that specialize to the groups $\bG(F)$ as $\bG$ runs over a family of reductive groups with a given absolute root datum; and in this family one constructs a family of definable subsets that specialize to the maximal parahorics in the corresponding groups $\bG(F)$.
It was previously shown in \cite{cluckers-gordon-halupczok:14b} that for all \emph{positive} $r$ and all \emph{optimal} points $x$ in the building,
the Moy-Prasad filtration subgroups ${\bG}(F)_{x, r}$ are definable in this sense; here we fill in the missing case where
$r=0$.\footnote{When $r=0$, the group ${\bG}(F)_{x,0}$ depends only on the facet containing $x$,
so we no longer need to consider optimal points $x$.}

\subsection{Reductive groups as a family of definable sets}
We treat reductive groups in the definable setting as in \cite{gordon-hales:15a}.
In fact, our construction of algebraic tori above in \S \ref{sub:def_tori} is a special case of this construction.
In particular, we have the \emph{fixed choices} that include the absolute root datum $\Xi$ of $\bG$, the Galois action on the absolute root datum (which we suppress from the notation), and a finite set $\cF$ that encodes
the set of facets of an alcove in the building of $\bG$, as in \cite{gordon-hales:15a}*{\S 2.1}.

More specifically, we start with a fixed choice  of a finite group $\Gamma$ as above, and an absolute root datum
$\Xi$ (which includes the action of $\Gamma$). This fixed choice determines  a split connected reductive group $\bG^{\ast\ast}$ defined over $\Q$,
with $\bT^{\ast\ast}\subset \bB^{\ast\ast}$ a split maximal torus and a Borel subgroup, and an action of $\Gamma$  on the root datum of $\bG^{\ast\ast}$
with respect to $(\bB^{\ast\ast}, \bT^{\ast\ast})$.  This determines a definable set $Z_{\Xi}^\ast$ that specializes,
for each local field $F$ of sufficiently large residue characteristic (with the bound depending only on $\Xi$)
to the set of pairs $z^\ast=(E, \zeta^\ast)$, where $E$ is a field extension of $F$ with Galois group isomorphic to 
$\Gamma$ (via an enumerated isomorphism), and $\zeta^\ast$ is an enumerated cocycle  with values in the group of outer automorphisms of $\bG^{\ast\ast}(F)$
defining a quasi-split $F$-form
$\bG_{z^\ast}$ of $\bG^{\ast\ast}$. Further, there is a definable set $Z_{\Xi}\to Z_{\Xi}^\ast$ encoding the inner $F$-forms of $\bG_{z^\ast}$ that become isomorphic to $\bG_{z^\ast}$ over $E$, (see\cite{gordon-hales:15a}*{\S 2.2.2} for details).


We also make a \emph{fixed choice} of a set $\cF$, the so-called `parahoric indexing set', defined precisely as in \cite{gordon-hales:15a}*{\S 2.1}
(We shall not need the details of its definition here, apart from the fact that it can be made a fixed choice).
For each $\Xi$, there are finitely many possible sets $\cF$ that could arise as the parahoric indexing set of a
reductive group with absolute root datum $\Xi$.
Our \emph{set of fixed choices} is now the set of pairs $(\Xi, \cF)$, with $\Xi$ as in the previous paragraph.
The correspondence between the cohomological data defining a group $\bG$ over $F$ and the indexing set for its
parahoric subgroups is described in \cite{gross:parahorics}*{\S 7}. From this explicit description, one can see that
there is a definable condition on an element $z=(z^\ast, \zeta)\in Z_\Xi(F)$ that ensures that
the reductive group $\bG$ over $F$ determined by the cocycle $z$ has the parahoric indexing set $\cF$, i.e.,
that $\cF$ can be identified with the set of baricentres of facets in an alcove in the building of $\bG$ over $F$.
Let us denote the subset of $Z_\Xi$ defined by this condition by $Z_{\Xi, \cF}$.

In summary, we have the following
\begin{theorem}\label{thm:def_groups}(\cite{gordon-hales:15a}*{\S 2.2.2})
 For every fixed choice $(\Xi, \cF)$ there exists $M>0$, definable sets $Z_{\Xi}^\ast$,
 $Z_{\Xi, \cF}\subset Z_{\Xi}$,
 and a definable family
 $\cG \to Z_{\Xi}$ such that for every $F\in \loc_M$ 
 the following holds:
 \begin{enumerate}
 \item  for $z\in Z_\Xi(F)$, the set ${\cG_z}(F)$ is the set of $F$-points of a connected reductive group
$\bG_z$ with absolute root datum determined by $\Xi$, or empty;
\item For each $z^\ast\in Z_{\Xi}^\ast(F)$ there exists an element which we will denote by $(z^\ast, 1)$ in $Z_{\Xi}(F)$, such that $\bG_{(z^\ast, 1)}$ is quasi-split over $F$;
\item If $z\in Z_{\Xi, \cF}$, then the facets in the alcove in the building for $\bG_z$ over $F$ are in bijection with the set $\cF$.
\end{enumerate}
Moreover, for $F\in \loc_M$, every isomorphism class of $F$-groups with absolute root datum given by $\Xi$ that split over a tamely ramified extension with Galois group $\Gamma$
arises as a fiber $\cG_z$ for some $z\in Z_\Xi(F)$.
\end{theorem}

\subsection{Definability of maximal parahorics}
Our main result is that the parahoric subgroups associated with facets in the building of $\bG$ via Bruhat-Tits theory are definable.
More precisely, let $\bG$ be a tamely ramified, connected reductive group defined over a local field $F$.
Let $\ff$ be a facet in the building of $\bG$.  The next proposition shows that
then the parahoric $\bG(F)_\ff\subset \bG(F)$ arises in a definable family of definable sets.
In particular, the canonical parahoric of $\bG(F)$ defined in \cite{gross:97a} is definable.

In order to prove the proposition, we need to start with  more fixed choices: namely, we have to include both the fixed choices needed to define the group and its parahoric indexing set as above, and also the fixed choice of a resolution $Y$ of the co-character lattice $X$ of a maximally split maximal torus in that group (which is part of $\Xi$), as in 
\S \ref{sub:fixedchoices}.

\begin{prop}\label{prop:main}
 Let $\Xi=(X,\Phi, X^\vee, \Phi^\vee, \Gamma)$,  $\cF$, $Z_\Xi\to Z_\Xi^\ast$ and $\cG\to Z_\Xi$ be, respectively, an absolute root datum with Galois action, a parahoric indexing set, the space encoding the forms of the split group with the given absolute root datum, and the definable family of all groups with this absolute root datum, as in Theorem \ref{thm:def_groups}, and let $Y\to X$ be a surjective map of $\Gamma$-modules. 

Then for each $\ff \in \cF$, there exists $M>0$ (depending only on the fixed choices ($\Xi, \cF$) and $Y\to X$) and a family of definable 
subsets $\cG_{\ff} \to Z_{\Xi, \cF}$ of $\cG$
 such that, for all $F\in \loc_M$ and $z\in Z_{\Xi, \cF}(F)$, 
\[
(\cG_{\ff})_{z}(F)= \bG_z(F)_{\ff}.
\]
\end{prop}

\begin{proof} Even though as above, one should start from the fixed choices, 
and then build a family of definable sets that specialize to the parahorics ${\cG_z(F)}_\ff$ for all fields $F$ of sufficiently large residue characteristic, we will just show how to construct the subgroups 
${\cG_z(F)}_\ff$ in a definable way pretending that $F$ and $z$ are fixed, in order not to clutter the discussion. It will be clear from the construction that this way we get a definable family of definable subsets as usual. 

First, consider the family of quasi-split groups $\bG_{z^\ast}$ parameterized by 
$z^\ast\in Z_{\Xi}^\ast$. 
Let $\bT^\ast$ be a maximal torus containing the maximal $F$-split torus in 
$\bG_{z^\ast}$ with co-character lattice isomorphic to $X$ (which is part of the fixed choices).  
Let $x$ be the baricentre of $\ff$.
By definition, $\bG(F)_{x,0}$ is generated by $\bT^\ast(F)_{x, 0}$ and $U_\psi$, where the $U_\psi$ are
the filtration subgroups of the unipotent one-parameter subgroups $U_{\alpha}$. 
We have shown in Proposition \ref{prop:NerCdefinable} above that  $\bT^\ast(F)_{x, 0}=\NerC{T^\ast}(F)$ is definable.
The rest of the proof in this case is identical to that of Lemma 3.4 in \cite{cluckers-gordon-halupczok:14b}. 
This proves the statement for quasi-split groups. 

Now suppose $z\in Z_\Xi$. The element $z$ in particular defines a tower of field extensions $E/L/F$ 
with $L/F$ a maximal unramified sub-extension of $E$, as in \S \ref{sub:def_tori}, a quasi-split group $\bG_{z^\ast}$,
 and a cocycle $\zeta$ that defines an inner twisting $\psi_z$ (over $F$) between the quasi-split form $\bG_{z^\ast}$ 
 and $\bG_z$, which is an $L$-isomorphism.  Note that $\psi_z$ is a definable map (using $z$ as a parameter). 
Let $\cF$ be such that $z\in Z_{\Xi, \cF}$, i.e., assume that the fixed choice $\cF$ provides the parahoric indexing 
set for $\bG_z$. Let $\ff$ be an element of $\cF$; we think of it as the baricentre of a facet in the building of 
$\bG_z$ over $F$. Since $L/F$ is unramified, the set of fixed points of the building of $\bG_z$ over $L$ 
under the $\gal(L/F)$-action is precisely the building of $\bG_z$ over $F$; in particular, we can view $\ff$ as a point 
in the building for $\bG_z$ over $L$, which coincides with the building of $\bG_{z^\ast}$ over $L$.  
As shown in the previous paragraph, $(\bG_{z^\ast}(L))_\ff$ is a definable set. 
Since $L/F$ is unramified, $\bG_{z}(F)_\ff$ is the set of fixed points under the action of $\gal(L/F)$ twisted by $\psi_z$ of the set  $(\bG_{z^\ast}(L))_\ff$, and thus it is definable. 
\end{proof} 

\subsection{Applications} 
As an immediate consequence of Proposition \ref{prop:main}, we obtain that the canonical measure is motivic, up to a motivic constant.
This statement was previously known for unramified reductive groups \cite{cluckers-hales-loeser}.
We recall that a \emph{motivic constant} is an element of the ring of constructible motivic functions on a point,
i.e., of $A:=\Z[\lef, \lef^{-1}, \frac{1}{1-\lef^{-i}}, i>0]$ where $\lef$ is a formal symbol which specializes to $q$.

For a connected quasi-split reductive group $\bG$ over a local field $F$, we write
$d\mu_{\bG(F)}^\can$ for the \emph{canonical Haar measure}
on $\bG(F)$, which assigns volume $1$ to the canonical parahoric. 
Note that this seems to be the standard definition of the canonical measure in all settings
where it is used to define global orbital integrals, but it differs from Gross' canonical measure
exactly by the $L$-factor of the motive associated with $\bG$ \cite{gross:97a}*{Prop. 4.7}.
For general $\bG$, using the same method as Gross, 
we define $d\mu_{\bG(F)}^\can$ as the pull-back of the canonical measure (in our sense)
from the quasi-split inner form $\bG^\ast$ of $\bG$.  

\begin{theorem}\label{thm:mot_meas}
 Let $\Xi$, $\cF$, $Z_\Xi$, $Z_{\Xi,\cF}$, $Z^\ast_\Xi$  and $\cG \to Z_\Xi$ be as in Theorem \ref{thm:def_groups}. 
Then there exists $M>0$ (depending only on the fixed choices), a family of  motivic measures
$d\mu_z^\mot$ on $\cG_z$, and a motivic function $c$ on $Z$ such that, for every $F\in \loc_M$,
\[
c_F(z) d\mu_{\bG_z(F)}^\can = d\mu_{z, F}^\mot.
\]
Here $d\mu_{\bG_z(F)}^\can$ is the canonical measure on $\bG_z(F)=\cG_z(F)$,
and $d\mu_{z, F}^\mot$ is the specialization to $F$ of the motivic measure $d\mu_z^\mot$ on the definable set $\cG_z$. 
\end{theorem} 

\begin{proof} 
We first define the motivic function $c^\ast$ on $Z^\ast$ that is responsible for scaling of the measure on quasi-split groups, and then define the motivic function $c$ on $Z$ by pull-back. 
There exists a motivic Haar measure on $\bG_{z^\ast}$ for every $z^\ast\in Z^\ast(F)$,
constructed e.g. in \cite{gordon-hales:15a}*{\S 2.3}. Let us denote this measure by $\mu_{z^\ast}^\mot$. 
Since the group $\bG_{z^\ast}$ is quasi-split,   
it has the canonical parahoric subgroup as in \cite{gross:97a}, associated with a special point $x$ in the building. The equivalence class of $x$ has a representative $\ff\in \cF$, and by Proposition \ref{prop:main}, the canonical parahoric 
$\bG_{z^\ast}(F)_{x, 0}=\bG_{z^\ast}(F)_\ff$ is definable.  Thus, we can define  a motivic function $c^\ast(z^\ast):=
\vol_{\mu_{z^\ast}^\mot}({\cG_{\ff}}_{z^\ast})$. 
By definition, we have 
\[
c^\ast_F(z^\ast) d\mu_{\bG_{z^\ast}(F)}^\can = d\mu_{z^\ast, F}^\mot.
\]

Now if $z=(z^\ast, \zeta)\in Z(F)$, the canonical measure on $\bG_z(F)$ is by definition the pull-back of the canonical measure on $\bG_{z^\ast}(F)$ under the inner twisting $\psi_{\zeta}$ determined by $\zeta$. 
The inner twisting is a definable map
(using $\zeta$ and $z^\ast$ as parameters; cf. \cite{cluckers-gordon-halupczok:14b}*{\S 3.5.2}). 
The pull-back of a motivic measure under a definable map is motivic.
Thus we have, by definition, 
$$ 
d\mu_{\bG_{z}(F)}^\can  =  
\psi_{\zeta}^\ast \left(d\mu_{\bG_{z^\ast}(F)}^\can \right).$$ 
The measure $c^\ast_F(z^\ast)\psi_{\zeta}^\ast \left(d\mu_{\bG_{z^\ast}(F)}^\can \right)$ is motivic, as a pull-back of a motivic measure under a definable map. 
Thus, for $z=(z^\ast, \zeta)\in Z$, we can define $c(z):= c^\ast(z^\ast)$, i.e., 
the function $c$ is simply the pull-back of the function $z^\ast$ on $Z^\ast$ under the projection map $Z\to Z^\ast$. 
\end{proof}

\begin{rem} The motivic measure $d\mu_{z, F}^\mot$ on $\bG(F)$ that we defined above differs (by a motivic constant) from the 
motivic measure defined in \cite{gordon-hales:15a}*{\S 2.3} and  \cite{cluckers-gordon-halupczok:14b} in the case when $\bG_z$ is not quasi-split over $F$. 
This definition allows a number of improvements in \cite{cluckers-gordon-halupczok:14b} and 
\cite{shin-templier:15a}*{Appendix B}. 
Namely,  parts (1) and (2) of \cite{cluckers-gordon-halupczok:14b}*{Lemma 3.4} become unnecessary, and the statement in Part (3)
now includes $r=0$; instead of using the measures discussed in \S 3.5.2, one can use the canonical measure. 
More importantly, in \cite{shin-templier:15a}*{Appendix B, \S B.5.2}, Definition 14.13 and Lemma 14.14
become unnecessary, and in all calculations one can take $i_M=1$, which simplifies the rest of the proof.
\end{rem}

Finally, Theorem \ref{thm:mot_meas} implies that formal degree in definable families of
supercuspidal representations is a motivic function of the parameters indexing the family. 
Recall that formal degree of a representation depends on the choice of Haar measure on 
$\bG(F)/C$, where $C$ is the center of $\bG(F)$. 
We use the canonical measure on $\bG^\ad$ in order to keep the statement simple and easy to use. 
More precisely, note that the difference between $\bG$ and $\bG^\ad$ lies in the fixed choices; let $\Xi$ be the fixed choices for 
$\bG$ and  $\Xi^\ad$ for $\bG^{\ad}$,  respectively. 
Then there is a map of the corresponding cocycle spaces $Z\to Z_{\Xi^\ad}$ (identity on the part of the parameter $z$ that determines the splitting field, and with the map on  cocycles induced by
the map $\bG^{\ast\ast} \to (\bG^{\ast\ast})^{\ad}$). By abuse of notation, we will denote the image of $z\in Z_\Xi(F)$ in 
$Z_{\Xi^\ad}(F)$ also by $z$.  
By Theorem \ref{thm:mot_meas} applied to the adjoint group $\bG_z^{\ad}$, we have the function $c^\ad$ on $Z_{\Xi^\ad}$ such that the measure 
$c^\ad(z) \mu_z^{\can}$ on the adjoint group with the parameter $z$ is motivic. 
 
\begin{cor} \label{cor:form_deg} Let $\Xi$ 
and $\cG \to Z_\Xi$ be as in Theorem \ref{thm:def_groups}, and let $M>0$ be the constant from Theorem \ref{thm:mot_meas}.
Suppose that we are given a definable family of compact subgroups $J_{z, \lambda} \subset \cG_z$,
parameterized by $z \in Z_\Xi$ and $\lambda$ in some definable set.
In addition, suppose we are given an irreducible representation  
$\sigma_{z, \lambda}(F)$ of $J_{z, \lambda}(F)$ of fixed dimension $d$, for every $F\in \loc_M$, such 
that $\pi_{z,\lambda}(F) := \ind_{J_{z,\lambda}(F)}^{\bG_z(F)} \sigma_{z,\lambda}(F)$
is irreducible for all $z \in Z_\Xi(F)$ and $F \in \loc_M$.
Then the formal degree (with respect to the canonical measure as above) of $\pi_{z, \lambda}$ is of the form 
\[
d_{\pi_{z, \lambda}}=\frac{c^\ad(z)d}{m(z, \lambda)},
\]
where $c^\ad(z)$ is defined in the above paragraph, and 
$m(z, \lambda)$ is a motivic function of $z, \lambda$. 
\end{cor}
\begin{proof}
It is well-known (see Lemma \ref{lem:well-known}) that
\begin{equation}\label{eq:well-known}
d_{\pi_{z, \lambda}}=\frac{\deg(\sigma_{z, \lambda})}{\vol_{\mu_{\bG_z^\ad(F)}^\can}(J_{z, \lambda}/C)}.
\end{equation}

By Theorem \ref{thm:mot_meas} applied to the adjoint group $\cG\to Z_{\Xi^\ad}$, we have 
\[
\vol_{\mu_{\bG_z^\ad(F)}^\can}(J_{z, \lambda}/C) = \frac{\vol_{\mu^\mot_z}(J_{z, \lambda})}{c^\ad(z)},
\]
where $\mu^\mot_z$ is the motivic measure, 
and the statement follows, with $m(z, \lambda)=\vol_{\mu^\mot_z}(J_{z, \lambda}/C)$.
\end{proof}

In the proof of Corollary \ref{cor:form_deg}, we used the following property of the formal degree.
This statement is well-known, but we could not find a reference, so we give a proof provided to us by Loren Spice.

\begin{lem} \label{lem:well-known}
If $G$ is a reductive $p$-adic group with center $C$, $\mu$ is a measure on $G/C$, $J \subseteq G$ is open, $J/C$ is compact,
$\sigma$ is an irreducible representation of $J$ and $\pi = \ind_J^G \sigma$ then the formal degree of $\pi$ with respect to $\mu$ is given by
\[
d_{\pi,\mu} = \frac{\deg(\sigma)}{\vol_\mu(J/C)}.
\]
\end{lem}
\begin{proof}
Let $V$ be the representation space of $\sigma$ and $(,)$ be a $\sigma$-invariant
inner product on $V$, which exists since $J$ is compact.
Choose a unit-length vector $v \in V$ and let $f$ be the function supported on $J$
defined there by $f(j) = (\sigma(j)v,v)$.  For the normalized Haar measure $dj$ on $J/C$,
\[
\int_{J/C} \lvert f(j)\rvert^2 dj = \frac{(v,v)}{\deg(\sigma)} = \frac{1}{d},
\]
and thus
\[
\frac{f(1)}{d_\pi} = \int_{G/C} \lvert f(g) \rvert^2 d\mu(g) = \frac{\vol_\mu(J/C)}{d}.
\]
\end{proof}

\subsubsection{Gross' canonical volume form}
We conclude with a remark about the canonical volume form and the motive of a reductive group. 
Here assume that $\bG$ is quasi-split. In  \cite{gross:97a}*{\S 4}, Gross defined the canonical volume form
$\omega_G$ on $G=\bG(F)$ associated with the canonical smooth model 
$\underline{G}^0$ of $G$ whose set of $\ri_F$-points is the canonical maximal parahoric subgroup. 
One can ask if this volume form itself gives rise to a motivic measure. 
A priori, the smooth group scheme $\underline{G}^0$ is defined over $\ri_F$, and thus it is not clear why the
associated volume form can be defined \emph{uniformly in $F$} using the Denef-Pas language.
However, we observe that the motive $M$ associated with $\bG$ can be determined directly from the fixed choices defining $\bG$,
since $\bG$ is quasi-split.
By \cite{gross:97a}*{Proposition 4.7}, the volume of the canonical parahoric with respect to $|\omega_G|$
is $L(M^\vee(1))$, which is a motivic constant.
In fact, $L(M^\vee(1))$ agrees with our motivic constant $c(z)$ when $z$ defines a split group, since
the motivic volume form we use to define the motivic measure coincides with $\omega_G$ in this case.
In the case that $\bG_z$ is not split, the question of whether $c(z) = L(M^\vee(1))$  is left open.

\bibliography{roebib/Biblio}


\end{document}